\documentclass[a4paper,11pt]{amsart}
\usepackage{amsfonts,amssymb,amsthm,cite,amsmath,amstext}
\usepackage{color}
\usepackage[colorlinks,linkcolor=black,citecolor=black]{hyperref}
\usepackage{comment}
\usepackage{framed}

\definecolor{shadecolor}{gray}{0.875}

\newtheorem{thrm}{Theorem}[section]
\newtheorem{thrmx}{Theorem}

\newtheorem{lem}[thrm]{Lemma}
\newtheorem{cor}[thrm]{Corollary}
\newtheorem{prop}[thrm]{Proposition}

\theoremstyle{definition}
\newtheorem{defn}[thrm]{Definition}

\newtheorem{exmple}[thrm]{Example}
\newtheorem{rmk}[thrm]{Remark}
\newtheorem{ques}[thrm]{Question}

\DeclareMathOperator{\vol}{vol}

\title{Hodge-index type inequalities, hyperbolic polynomials and complex Hessian equations}
\author{Jian Xiao}
\date{}

\begin{document}
\maketitle

\begin{abstract}
It is noted that using complex Hessian equations and the concavity inequalities for elementary symmetric polynomials implies a generalized form of Hodge index inequality.
Inspired by this result, using G{\aa}rding's theory for hyperbolic polynomials, we obtain a mixed Hodge-index type theorem for classes of type $(1,1)$.
The new feature is that this Hodge-index type theorem holds with respect to mixed polarizations in which some satisfy particular positivity condition, but could be degenerate and even negative along some directions.
\end{abstract}

\tableofcontents

\section{Introduction}
\subsection{Classical Hodge index theorem}
The classical Hodge index theorem for an algebraic surface determines the signature of the intersection pairing on the algebraic curves. More precisely, on the algebraic surface the space spanned by the numerical classes of curves has a one-dimensional subspace (not uniquely determined) on which it is positive definite, and decomposes as a direct sum of some such one-dimensional subspace, and a complementary subspace on which it is negative definite.
In higher dimensions, we have the following Hodge index theorem for $(1,1)$-classes on a compact K\"ahler manifold $X$ of dimension $n$, which is a particular case of the Hodge-Riemann bilinear relations. Let $\omega$ be a K\"ahler class on $X$, then we could define the so called primitive space of $(1,1)$-classes with respect to $\omega$:
\begin{equation*}
  P^{1,1}(X, \mathbb{C})=\{\gamma \in H^{1,1} (X, \mathbb{C})| \omega^{n-1} \cdot \gamma =0\}.
\end{equation*}
It is clear that $P^{1,1}(X, \mathbb{C})$ is a hyperplane in $H^{1,1} (X, \mathbb{C})$. On $H^{1,1} (X, \mathbb{C})$, one has the following quadratic form
\begin{equation*}
  Q(\alpha, \beta) = \alpha \cdot \overline{\beta}\cdot \omega^{n-2}.
\end{equation*}
Then the Hodge index theorem for $(1,1)$-classes says that, for any $\alpha \in P^{1,1}(X, \mathbb{C})$, we have $Q(\alpha, \alpha) \leq 0$. We call this inequality the Hodge-index inequality (with respect to $\omega$). Furthermore, the equality holds if and only if $\alpha =0$.

There are several approaches to the classical Hodge index theorem or the more general Hodge-Riemann bilinear relations. For example, it can be proved by reducing the global case to the local case by harmonic forms (see e.g. \cite{voisinHodge1}), or by using the deep relationship between polarized Hodge-Lefschetz modules and variations of Hodge structures (see \cite{cattanimixedHRR}), or by reducing to the local case by applying the $L^2$-method to solve a $dd^c$-equation (see \cite{DN06}). Note that the later two proofs apply to the more general mixed situation.

\subsection{The main result}
In this note, we present relations between Hodge-index type inequalities, hyperbolic polynomials and the solvability of some PDEs in complex geometry.

As the starting point, we observe that the classical Hodge-index inequality (with respect to $\omega$) follows easily from Yau's solution to complex Monge-Amp\`{e}re equations (i.e., the Calabi-Yau theorem). Yau's theorem enables us to reduce it to the local case, then we could use the concavity or hyperbolicity of determinants.

Note that the complex Monge-Amp\`{e}re equation is essentially a PDE with respect to a K\"ahler class (or a K\"ahler metric) $\omega$. Roughly speaking, the philosophy in the above approach can be stated as follows:
\begin{itemize}
  \item the positivity condition on $\omega$ ensures the solvability of the complex Monge-Amp\`{e}re equation with respect to $\omega$;
  \item we could associate the primitive subspace $P^{1,1}$ to $\omega^{n-1}$ and the quadratic form $Q$ to $\omega^{n-2}$;
  \item combining with the properties of determinants, the signature of $Q$ (or at least the inequality) follows from the solvability of complex Monge-Amp\`{e}re equations.
\end{itemize}

By using the same philosophy, we find that the positivity condition in the solvability of complex Hessian equations works well along the above framework. This enables us to somehow weaken the positivity condition in the definitions of primitive spaces and the corresponding quadratic forms.

Let us recall the positivity condition in complex Hessian equations. Let $\omega$ be a K\"ahler metric on $X$, and let $\widehat{\alpha}$ be a real $(1,1)$-form, then $\widehat{\alpha}$ is called \emph{$m$-positive} ($1\leq m \leq n$) with respect to $\omega$, if the following inequalities hold at every point:
\begin{equation*}
  \widehat{\alpha}^k \wedge \omega^{n-k} >0,\ \forall \ k = 1, 2,...,m.
\end{equation*}
We denote the set of smooth $m$-positive $(1,1)$-forms by $\Gamma_m$. (In order to simplify the notations, we omit the reference K\"ahler metric $\omega$.) $\Gamma_m$ is an open convex cone in the vector space of real $(1,1)$-forms. It is clear that in general we have
\begin{equation}\label{eq pos inclusion}
  \Gamma_n \subsetneq \Gamma_{n-1} \subsetneq ... \subsetneq \Gamma_1.
\end{equation}
It is well-known that $\Gamma_n$ is the set of K\"ahler metrics.
For $(1,1)$-classes, we call a $d$-closed $(1,1)$-class $\alpha$ $m$-positive with respect to the K\"ahler metric $\omega$, if it has an $m$-positive representative in the pointwise sense. We use the same symbol ${\Gamma}_m \subset H^{1,1} (X, \mathbb{R})$ to denote the set of $m$-positive $(1,1)$-classes on $X$. Then it is an open convex cone in $H^{1,1} (X, \mathbb{R})$. The sequence of cones ${\Gamma}_m$ satisfies the same inclusion relation as (\ref{eq pos inclusion}), and $\Gamma_n$ is exactly the K\"ahler cone of $X$.

By using complex Hessian equations and the concavity inequalities for elementary polynomials, we get a generalized form of Hodge index inequalities for $(1,1)$-classes.
Inspired by this result, we find that using G{\aa}rding's theory for hyperbolic polynomials and solving Laplacian equations imply a mixed Hodge-index type theorem.

Let us make a convention: in this paper, we will always use the same symbol $\omega$ to denote the class of a K\"ahler metric or a K\"ahler metric in the class.

Our main result is the following:

\begin{thrmx}\label{thrm main result}
Let $X$ be a compact K\"ahler manifold of dimension $n$, and let $\omega$ be a reference K\"ahler metric. Let $\alpha_1, ...,\alpha_{m-1} \in H^{1,1}(X, \mathbb{R})$, and assume that every $\alpha_j$ is $m$-positive with respect to $\omega$. Denote $\Omega=\alpha_1 \cdot \alpha_2 \cdot ...\cdot \alpha_{m-2}$.
Let
\begin{equation*}
  P^{1,1}(X, \mathbb{C})=\{\gamma \in H^{1,1} (X, \mathbb{C})| \omega^{n-m}\cdot \Omega \cdot \alpha_{m-1} \cdot \gamma =0\}
\end{equation*}
be the primitive subspace defined by $\omega^{n-m}\cdot \Omega \cdot \alpha_{m-1}$.
Then we have:
\begin{itemize}
  \item The quadratic form
\begin{equation*}
  Q(\beta, \gamma) = \beta \cdot \overline{\gamma}\cdot \Omega \cdot \omega^{n-m}.
\end{equation*}
is negative definite on $P^{1,1}(X, \mathbb{C})$.

  \item The space $H^{1,1} (X, \mathbb{C})$ has a direct sum decomposition
  \begin{equation*}
    H^{1,1} (X, \mathbb{C}) = P^{1,1}(X, \mathbb{C}) \oplus \mathbb{C} \alpha_{m-1}.
  \end{equation*}

  \item The map $\omega^{n-m}\cdot \Omega: H^{1,1} (X, \mathbb{C}) \rightarrow H^{n-1,n-1} (X, \mathbb{C})$ is an isomorphism.
\end{itemize}
\end{thrmx}

If $\alpha_1 =...=\alpha_{m-1}= \omega$, we get the classical Hodge index theorem for $(1,1)$-classes.

\begin{exmple}
An interesting example is given by a holomorphic submersion. Let $f: X\rightarrow Y$ be a holomorphic submersion from a compact K\"ahler manifold $X$ of dimension $n$ to a compact K\"ahler manifold $Y$ of dimension $m$. Let $\omega_X$ be a K\"ahler class on $X$, and let $\omega_{Y_1},...,\omega_{Y_{m-1}}$ be K\"ahler classes on $Y$. Then the Hodge-index type inequalities hold for $f^* \omega_{Y_1},...,f^*\omega_{Y_{m-1}}, \omega_X$. Note that $f^*\omega_{Y_j}$ is degenerate along the fibers. In some sense, this could be considered as a ``relative'' version of the Hodge index theorem.
\end{exmple}

\begin{rmk}
Theorem \ref{thrm main result} indicates that a more general Hodge-Riemann bilinear relations should hold with respect to mixed $m$-positive classes coupled with a K\"ahler class (see Section \ref{sec HRR}). This will be discussed elsewhere.
\end{rmk}

As an immediate corollary of Theorem \ref{thrm main result}, we get the following log-concavity result.

\begin{thrmx}\label{KT m-posiitve0}
Let $X$ be a compact K\"ahler manifold of dimension $n$, and let $\omega$ be a reference K\"ahler metric on $X$. Assume that $\alpha, \beta \in H^{1,1}(X, \mathbb{R})$ are $m$-positive with respect to $\omega$. Denote $a_k = \alpha^k \cdot  \beta^{m-k} \cdot \omega^{n-m}$, where $0\leq k \leq m$. Then the sequence $\{a_k\}$ is log concave, that is, $$a_k ^2 \geq a_{k+1} a_{k-1}$$ for any $1\leq k \leq m-1$. Moreover, $a_k ^2 = a_{k+1} a_{k-1}$ for some $k$ if and only if $\alpha, \beta$ are proportional.
\end{thrmx}

In Section \ref{sec prel}, we give an overview on hyperbolic polynomials and complex Hessian equations. In Section \ref{sec hessian to hodgeindex}, we first present the implication from complex Hessian equations to a form of Hodge index inequalities, which motivates Theorem \ref{thrm main result}. We then prove the main result and give some further discussions around our Hodge-index type theorem.

\subsection*{Acknowledgements}
We would like to thank Jie Liu for some interesting discussions.

\section{Preliminaries}\label{sec prel}

\subsection{Hyperbolic polynomials}\label{sec hyperb}
We give a brief review on hyperbolic polynomials, in particular, G{\aa}rding's concavity inequalities for such polynomials. Our references are \cite[Chapter 2]{hormanderConvexity} and G{\aa}rding's classical paper \cite{gardinghyperbolic}.

Let $V$ be a complex vector space of dimension $n$, and let $P = P(x)$ be a homogeneous polynomial of degree $m>0$ on $V$.

\begin{defn}
Let $a \in V$ be a real vector. We say that $P$ is \emph{hyperbolic} at $a$, if the equation $P(sa +x)=0$ (as a polynomial equation of $s$) has only real roots for every real vector $x$.
\end{defn}

From its definition, the hyperbolicity of $P$ at $a$ implies that $P(a)\neq 0$ and $\frac{P(x)}{P(a)}$ is real whenever $x$ is real. Thus a hyperbolic polynomial is essentially real.

\begin{defn}
The \emph{linearity} $LP$ of $P$ is defined as the set of all $x$ such that $P(tx + y) =P(y)$ for all $t, y$. Then $LP$ is a linear subspace.
We say that $P$ is \emph{complete}, if $LP = \{0\}$.
\end{defn}

For hyperbolic polynomials, we have:

\begin{thrm}
Assume that $P$ is hyperbolic at $a$. Let $\Gamma (P, a)$ be the component containing $a$ of the set
$\{x \in V |\ P(x) \neq 0\}$.
Then $\Gamma (P, a)$ is an open convex cone. For any $b \in \Gamma (P, a)$, $P$ is hyperbolic at $b$ and $\Gamma (P, a) =\Gamma (P, b)$. Moreover, the polynomial $\frac{P(x)}{P(a)} >0$ when $x \in \Gamma (P, a)$, and $\left(\frac{P(x)}{P(a)}\right)^{1/m}$ is concave and homogeneous of degree one on $\Gamma (P, a)$, and is zero on the boundary of $\Gamma (P, a)$.
\end{thrm}

We call $\Gamma (P, a)$ the positive cone associated to $P$.

Let $P(x^1, ..., x^m)$ be the completely polarized form of the polynomial $P$, that is,
\begin{equation*}
  P(x^1, ..., x^m) = \frac{1}{m!} \prod_k (\sum_j x_j ^k \frac{\partial}{\partial x_j}) P(x),
\end{equation*}
where the $x^k= (x_1 ^k, x_2 ^k, ..., x_n ^k) \in V$. The completely polarized form is characterized by being linear in each argument, invariant under permutations and satisfying $P(x,...,x) =P(x)$. It is clear that, for any fixed $x^1, ...,x^l \in V$ $(l<m)$, $P(x^1,...,x^l, x,...,x)$ is a homogeneous polynomial of degree $m-l$. Furthermore, we also have:

\begin{thrm}\label{thrm garding positive}
Assume that $P$ is hyperbolic at $a$. Then for any fixed $b^1, ...,b^l \in \Gamma (P, a)$ $(l<m)$, $P_l (x):=P(b^1,...,b^l, x,...,x)$ is hyperbolic at $a$ and $\Gamma (P, a) \subset \Gamma (P_l, a)$. Moreover, if $P$ is complete at $a$, then $P_l$ is also complete at $a$, when $m-l \geq 2$. If $P$ is complete, then $P(x^1, ...,x^m)>0$ for any $x^1 \in \overline{\Gamma (P, a)}\setminus \{0\}, x^2,...,x^m \in \Gamma (P, a)$.
\end{thrm}

Now we could state the well known G{\aa}rding's inequality for hyperbolic polynomials.

\begin{thrm}\label{thrm garding ineq}
Assume that $P$ is hyperbolic at $a$. Then for any $x^1, ...,x^m \in \Gamma (P, a)$, we have
\begin{equation}\label{eq garding ineq}
  P(x^1,...,x^m) \geq P(x^1)^{1/m}\cdot...\cdot P(x^m)^{1/m}
\end{equation}
with equality if and only if the $x^j$ are pairwise proportional modulo $LP$.
In particular, if $P$ is complete, then the equality in (\ref{eq garding ineq}) holds if and only if the $x^j$ are pairwise proportional.
\end{thrm}

Next we present the well known applications of G{\aa}rding's theory to elementary symmetric polynomials.
Recall that the $m$-th elementary symmetric polynomial $\sigma_m$ of $n$ variables is defined by
\begin{equation*}
  \sigma_m (\lambda_1,...,\lambda_n) = \sum_{1\leq i_1 <...<i_m \leq n} \lambda_{i_1}\cdot...\cdot \lambda_{i_m}.
\end{equation*}

\begin{exmple}\label{exmple symmetric}
\begin{enumerate}
  \item The elementary symmetric polynomial $\sigma_m (\lambda)$ is hyperbolic at $(1,...,1)$ and complete, and the corresponding positive cone $\Gamma_m$ is given by
      \begin{equation*}
        \Gamma_m = \{x\in \mathbb{R}^n | \sigma_l (x)>0,\ \forall 1\leq l\leq m\}.
      \end{equation*}
  \item Let $A$ be an $n \times n$ Hermitian matrix, then $\sigma_m (A)$ is defined to be $\sigma_m (\lambda(A))$, where $\lambda(A) = (\lambda_1 (A),...,\lambda_n (A))$ is the list of eigenvalues of $A$. As a polynomial on the space of $n\times n$ Hermitian matrices $\mathcal{H}$, $\sigma_m$ is hyperbolic at the identity matrix and complete, and the corresponding positive cone $\Gamma_m$ is given by
      \begin{equation*}
        \Gamma_m = \{M\in \mathcal{H} | \sigma_l (M)>0,\ \forall 1\leq l\leq m\}.
      \end{equation*}

\item In our setting, it is convenient to translate the above examples to real $(1,1)$-forms. Let $$\omega = i \sum_{1\leq j,k \leq n} \omega_{jk} dz^j \wedge d\bar{z}^k$$ be a fixed strictly positive $(1,1)$-form with constant coefficients on $\mathbb{C}^n$, i.e., $[\omega_{jk}]$ is a positive definite Hermitian matrix. Denote the space of real $(1,1)$-form on $\mathbb{C}^n$ with constant coefficients by
$\Lambda^{1,1} _\mathbb{R} (\mathbb{C}^n)$.
Let $\widehat{\alpha} \in \Lambda^{1,1} _\mathbb{R} (\mathbb{C}^n)$, then $\sigma_m (\widehat{\alpha})$ is defined by
\begin{equation*}
  \sigma_m (\widehat{\alpha}) = \widehat{\alpha}^m \wedge \omega^{n-m}.
\end{equation*}
Then we have: $\sigma_m$ is hyperbolic at $\omega$ and complete, and the corresponding positive cone $\Gamma_m$ is given by
\begin{equation*}
\Gamma_m = \{\widehat{\alpha}\in \Lambda^{1,1} _\mathbb{R} (\mathbb{C}^n) | \sigma_l (\widehat{\alpha})>0,\ \forall 1\leq l\leq m\}.
\end{equation*}

\end{enumerate}
\end{exmple}

\subsection{Complex Hessian equations}
In this section, we give some discussions on $m$-positivity and complex Hessian equations.
We assume that $X$ is a compact K\"ahler manifold of dimension $n$. Let $\omega$ be a K\"ahler metric on $X$. In order to introduce complex Hessian equations, let us recall from the introduction:

\begin{defn}
Let $\widehat{\alpha}$ be a smooth real $(1,1)$-form on $X$, then $\widehat{\alpha}$ is called $m$-positive with respect to $\omega$ if $\widehat{\alpha}^k \wedge \omega^{n-k} >0$ for every $1\leq k \leq m$ and every point on $X$.
\end{defn}

From its definition, it is clear that $m$-positivity is defined by the positive cone associated to the hyperbolic polynomial $\sigma_m$.

\begin{exmple}
By considering the linear case, i.e., $m$-positivity $(m<n)$ on a torus, it can be seen that an $m$-positive $(1,1)$-form can be degenerate and even negative along some directions.
\end{exmple}

\begin{defn}
Let $\alpha \in H^{1,1} (X, \mathbb{R})$. We call $\alpha$ $m$-positive with respect to the K\"ahler metric $\omega$, if $\alpha$ has an $m$-positive representative.
\end{defn}

The positive cone $\Gamma_n$ of $(1,1)$-classes is exactly the K\"ahler cone, thus by \cite{DP04} we have a numerical characterization of $\Gamma_n$. For general $\Gamma_m \subset H^{1,1} (X, \mathbb{R})$ $(1< m <n)$, it is unclear what kind of numerical conditions would imply the existence of $m$-positive forms.

\begin{rmk}
For $m=1$, it is easy to see that $\alpha$ is 1-positive if and only if $\int \alpha \wedge \omega^{n-1} >0$. Let $\widehat{\alpha}$ be a smooth representative of $\alpha$, and let $\Phi >0$ be a smooth volume form satisfying $$\int \Phi = \int \alpha \wedge \omega^{n-1}.$$
Then by the solvability of Laplacian equations, we could find a smooth function $\phi$ satisfying $$(\widehat{\alpha} + dd^c \phi)\wedge \omega^{n-1} = \Phi >0.$$
Thus $\widehat{\alpha} + dd^c \phi$ is 1-positive with respect to $\omega$. As a byproduct, let $\omega + dd^c \psi$ be a K\"ahler metric in the same class of $\omega$, then the $(1,1)$-class $\alpha$ is 1-positive with respect to $\omega$ if and only if it is 1-positive with respect to $\omega + dd^c \psi$. It is unclear to us if this holds for general $m$-positivity:

\begin{ques}
Assume that the $d$-closed $(1,1)$-form $\widehat{\alpha}$ is $m$-positive $(1< m<n)$ with respect to $\omega$, and let $\omega + dd^c \psi$ be another K\"ahler metric. Then does there exist a smooth function $\phi$ such that $\widehat{\alpha} + dd^c \phi$ is $m$-positive with respect to $\omega + dd^c \psi$?
\end{ques}

\end{rmk}

At last, let us recall the following fundamental result.
Suppose that $\widehat{\alpha}$ is a $d$-closed $m$-positive $(1,1)$-form, then for any strictly positive volume form $\Phi$ satisfying
$$\int_X \Phi = \int_X \widehat{\alpha}^m \wedge \omega^{n-m},$$
one could solve the following equation (see \cite{dinewHessian}):
\begin{equation*}
  (\widehat{\alpha} + dd^c \phi)^m \wedge \omega^{n-m} = \Phi
\end{equation*}
such that $\widehat{\alpha} + dd^c \phi$ is an $m$-positive $(1,1)$-form. In particular, when $m=n$ this is exactly Yau's solution to the complex Monge-Amp\`{e}re equations (see \cite{Yau78}).

\begin{rmk}
The above $m$-positivity is defined with respect to a single K\"ahler metric $\omega$. More generally, one could also consider a mixed version of $m$-positivity. Let $\omega_1, ..., \omega_{n-m} \in \Lambda_\mathbb{R} ^{1,1} (\mathbb{C}^n)$ be K\"ahler metrics, we call a form $\widehat{\alpha} \in \Lambda_\mathbb{R} ^{1,1} (\mathbb{C}^n)$ $m$-positive with respect to $\omega_1, ..., \omega_{n-m}$, if
\begin{equation*}
 \widehat{ \alpha}^k \wedge \omega_{i_1}\wedge...\wedge \omega_{i_{n-k}} >0
\end{equation*}
for any $1\leq k \leq m$ and any $1 \leq i_j \leq n-m$. A similar complex Hessian equation could also be proposed with respect to this mixed version of $m$-positivity. It might have some applications for this mixed $m$-positivity.
\end{rmk}

\section{Hodge-index type inequalities}\label{sec hessian to hodgeindex}
\subsection{Motivation results}
We first give a brief explanation on how complex Monge-Amp\`{e}re equations can be used to deduce the classical Hodge index inequalities, by using concavity.

For simplicity, we focus on the surface case. We start with the concavity property of determinants. Consider the function
\begin{equation*}
f(t) = \det (A+ tB)^{1/2},\ t> 0,
\end{equation*}
where $A, B$ are two positive definite Hermitian $2\times 2$ matrices. It is well known that $f(t)$ is concave. Then a direct calculation on $f''(t)$ and taking limit $t \rightarrow 0$ show that
\begin{equation*}
  \det(A, B)^2 \geq \det(A) \det(B),
\end{equation*}
where $\det(A,B)$ is the mixed determinant of $A, B$. It can be also shown that the equality holds if and only if $A = cB$ for some constant $c>0$.
Let $X$ be a compact K\"ahler surface, and let $\omega, \alpha$ be two K\"ahler classes on $X$. In \cite{gromov1990convex} and \cite{Dem93},
it was noted that applying directly Yau's solution to the Calabi conjecture \cite{Yau78} implies the concavity of $t\mapsto \vol(\omega + t \alpha)^{1/2}, t>0$. Similar to the above linear case, this yields
the following inequality:
\begin{equation}\label{eqhodgeindex surface}
  (\omega\cdot \alpha)^2 \geq \omega^2 \alpha^2.
\end{equation}
The equality holds if and only if $\omega, \alpha$ are proportional.
In their arguments, the complex Monge-Amp\`{e}re equation is used to reduce the global case to the local case.
For any $\beta \in H^{1,1}(X, \mathbb{R})$, by considering the quadratic form $Q(\beta) = (\omega \cdot \beta)^2 - \omega^2 \beta^2$, it is not hard to see the above inequality implies that: if $\omega \cdot \beta=0$, then $\beta^2 \leq 0$ with equality if and only if $\beta =0$. This is exactly the Hodge index theorem on a surface.

By similar arguments, the classical Hodge index theorem for $(1,1)$-classes (at least the inequalities) in higher dimensions can be also proved.

More generally, we show how complex Hessian equations can be used to get a generalized form of Hodge index inequalities for $(1,1)$-classes, which looks interesting and 
serves as the motivation for our main result -- Theorem \ref{thrm main result}.

In the following, we fixed a K\"ahler metric $\omega$ on $X$. We denote the set of real $(1,1)$-classes which have $m$-positive smooth representatives with respect to $\omega$ by $\Gamma_m$. Then $\Gamma_m \subset H^{1,1}(X, \mathbb{R})$ is an open convex cone. Recall from the introduction that we use the same symbol $\omega$ to denote a K\"ahler metric or a K\"ahler class.

\begin{lem}\label{lem concave ineq}
Let $\sigma_m (\alpha) = \alpha^m \cdot \omega^{n-m}$, where $\alpha\in H^{1,1}(X, \mathbb{R})$. Let $\sigma_m (\alpha_1,...,\alpha_m)=\alpha_1 \cdot ...\cdot \alpha_m \cdot \omega^{n-m}$ be the complete polarization of $\sigma_m$.
Then we have:
\begin{itemize}
  \item For any $\alpha_1,...,\alpha_m \in \Gamma_m$,
  \begin{equation*}
    \sigma_m (\alpha_1,...,\alpha_m) \geq \sigma_m ^{1/m} (\alpha_1)...\sigma_m ^{1/m}(\alpha_m)
  \end{equation*}
  with the equality holds if and if the $\alpha_k$ are pairwise proportional.
\end{itemize}
\end{lem}

\begin{proof}
The argument is inspired by \cite{gromov1990convex} and \cite{Dem93}, which reduces the global case to the local one.
By using complex Hessian equations, for every $k$, there exists an $m$-positive smooth representative $\widehat{\alpha}_k$ in the class of $\alpha_k$ solving
\begin{equation}\label{eq hessian}
  \widehat{\alpha}_k ^m \wedge \omega^{n-m} = c_k \omega^n,
\end{equation}
where $c_k = \int \alpha^m \wedge \omega^{n-m} / \int \omega^n$ is a positive constant.

Next we estimate $\sigma_m (\alpha_1,...,\alpha_m)$, by applying G{\aa}rding's inequality in the pointwise setting. We have
\begin{align*}
  \sigma_m (\alpha_1,...,\alpha_m) &= \int_X \widehat{\alpha}_1 \wedge ...\wedge \widehat{\alpha}_m \wedge \omega^{n-m}\\
  &\geq \int_X \prod_{k=1} ^m \left(\frac{\widehat{\alpha}_k ^m \wedge \omega^{n-m}}{\omega^n}\right)^{1/m} \omega^n\\
  &= \prod_{k=1} ^m c_k ^{1/m} \int_X \omega^n\\
  &= \prod_{k=1} ^m \sigma_m (\alpha_k)^{1/m},
\end{align*}
where in the above second inequality we have applied G{\aa}rding's inequality (see Theorem \ref{thrm garding ineq} and Example \ref{exmple symmetric}), and in the third equality we applied the Hessian equations (\ref{eq hessian}).

By the above estimates, it is clear that the equality in $\sigma_m (\alpha_1,...,\alpha_m) \geq \prod_{k=1} ^m \sigma_m (\alpha_k)^{1/m}$ holds if and only if we have equalities everywhere, in particular, we have
\begin{equation*}
  \widehat{\alpha}_1 \wedge ...\wedge \widehat{\alpha}_m \wedge \omega^{n-m}  =  \prod_{k=1} ^m \left(\frac{\widehat{\alpha}_k ^m \wedge \omega^{n-m}}{\omega^n}\right)^{1/m} \omega^n.
\end{equation*}
Applying Theorem \ref{thrm garding ineq} and Example \ref{exmple symmetric} again, the equality yields that the forms $\widehat{\alpha}_k$ are pairwise proportional at every point. We claim that this implies the classes $\alpha_k$ are pairwise proportional on $X$, by using the equations in (\ref{eq hessian}). Take $\widehat{\alpha}_1, \widehat{\alpha}_2$ for example, for any two points $p, q$ on $X$ there are two positive constants $c(p), c(q)$ such that
\begin{equation}\label{eq prop}
  \widehat{\alpha}_1 (p) = c(p) \widehat{\alpha}_2 (p), \ \  \widehat{\alpha}_1 (q) = c(q) \widehat{\alpha}_2 (q).
\end{equation}
Substituting the equalities in (\ref{eq prop}) to (\ref{eq hessian}) and noting that $c_1, c_2$ are constant functions on $X$, we get
\begin{equation*}
  c(p)^m = c(q)^m = c_1 /c_2.
\end{equation*}
Thus, $\widehat{\alpha}_1 (p) = \left(\frac{c_1}{c_2}\right)^{1/m} \widehat{\alpha}_2 (p)$ for any point $p$, which yields $$\alpha_1 = \left(\frac{c_1}{c_2}\right)^{1/m} \alpha_2$$
in $H^{1,1}(X, \mathbb{R})$.

This finishes the proof.

\end{proof}

\begin{rmk}
By Example \ref{exmple symmetric}, the function $\sigma_m$ defined on $\Lambda_{\mathbb{R}} ^{1,1} (\mathbb{C}^n)$ is complete and hyperbolic. By the classical Hodge index theorem, it is not hard to see that $\sigma_m$ defined on $H^{1,1}(X, \mathbb{R})$ is complete. However, except for $m=2$
(see the discussions in Section \ref{sec hessian to hodgeindex}), it is unclear to us whether $\sigma_m$ defined on $H^{1,1}(X, \mathbb{R})$ is hyperbolic. (In the general case, we suspect that it is not hyperbolic.)
\end{rmk}

\begin{lem}\label{lem concave ineq1}
Let $\sigma_m (\alpha) = \alpha^m \cdot \omega^{n-m}$, where $\alpha\in H^{1,1}(X, \mathbb{R})$.
Then we have:
\begin{itemize}
  \item The function $\sigma_m ^{1/m}$ is strictly concave on $\Gamma_m$ in the following sense: for any $\alpha, \beta \in \Gamma_m$ which are not proportional, the function $g(t) = \sigma_m ^{1/m} (\alpha + t \beta)$ is strictly concave when $t>0$.
\end{itemize}
\end{lem}

\begin{proof}
This follows immediately from Lemma \ref{lem concave ineq}. For any $s \in (0,1)$ and $t_1, t_2 >0$, we have:
\begin{align*}
  g(st_1 + (1-s)t_2) &= \left(
  \sum_{k=0} ^m \frac{m!}{k!(m-k)!}
   s^k (1-s)^{m-k} \sigma_m ((\alpha + t_1 \beta)^k, (\alpha + t_2 \beta)^{m-k})                                        \right)^{1/m}\\
  &> \left(\sum_{k=0} ^m \frac{m!}{k!(m-k)!}
s^k (1-s)^{m-k} \sigma_m (\alpha + t_1 \beta)^{k/m} \sigma_m (\alpha + t_2 \beta)^{m-k/m}                                        \right)^{1/m}\\
&= sg(t_1) + sg(t_2),
\end{align*}
where we applied Lemma \ref{lem concave ineq} in the second inequality.
\end{proof}

Starting with a homogeneous polynomial with concavity, the following arguments should be well-known. Here, we follow \cite{hormanderConvexity}.
By Lemma \ref{lem concave ineq1}, we get the following inequality.

\begin{lem}\label{lem hodge index pos}
Let $\alpha, \beta \in \Gamma_m$, then we have
\begin{align}\label{eq hodge index pos}
  \left(\beta \cdot \alpha\cdot \alpha^{m-2} \cdot \omega^{n-m} \right)^2
  \geq \left(\beta^2 \cdot \alpha ^{m-2} \cdot \omega^{n-m} \right)\left(\alpha^2 \cdot \alpha^{m-2} \cdot \omega^{n-m}\right).
\end{align}
\end{lem}

\begin{proof}

By Lemma \ref{lem concave ineq1}, the function
$$g(t) = \sigma_m ^{1/m} (\alpha + t \beta)$$
is concave. Thus, $g''(t)\leq 0$ for any $t>0$. A straightforward calculation shows that:
\begin{equation*}
  g'(t) = \left((\alpha + t\beta)^m \cdot \omega^{n-m}\right)^{\frac{1}{m} -1} \left(\beta \cdot  (\alpha + t\beta)^{m-1} \cdot \omega^{n-m}\right),
\end{equation*}
which yields
\begin{align*}
  g''(t) =& (m-1)\left( (\alpha + t\beta)^{m} \cdot \omega^{n-m}\right)^{\frac{1}{m} -1} \left(\beta^2 \cdot (\alpha + t\beta)^{m-2} \cdot \omega^{n-m} \right)\\
   &- (m-1) \left( (\alpha + t\beta)^{m} \cdot \omega^{n-m}\right)^{\frac{1}{m} -2} \left(\beta \cdot (\alpha + t\beta)^{m-1} \cdot \omega^{n-m} \right)^2.
\end{align*}
Using $g''(t)\leq0$ for any $t>0$ implies
\begin{align*}
  &\left(\beta \cdot (\alpha + t\beta)\cdot (\alpha + t\beta)^{m-2} \cdot \omega^{n-m} \right)^2\\
  &\geq \left(\beta^2 \cdot (\alpha + t\beta)^{m-2} \cdot \omega^{n-m} \right)\left((\alpha + t\beta)^2 \cdot (\alpha + t\beta)^{m-2} \cdot \omega^{n-m}\right).
\end{align*}


Taking a limit $t\rightarrow 0$ finishes the proof of Lemma \ref{lem hodge index pos}.
\end{proof}

\begin{rmk}
In the proof of Lemma \ref{lem hodge index pos}, we use the fact that $g(t)$ is concave. By Lemma \ref{lem concave ineq1}, it is even strictly concave when $\alpha, \beta$ are not proportional, however, this does not imply directly that $g''(t)<0$ for any $t>0$ (even though it is, see Remark \ref{rmk strict conc}). Anyway, when $\alpha, \beta$ are not proportional, the strict concavity implies that $g''(t)$ can not vanish on any sub-intervals of $\mathbb{R}_+$.
\end{rmk}

Lemma \ref{lem hodge index pos} can be generalized to the following form.

\begin{lem}\label{lem hodge index pos1}
Let $\alpha\in \Gamma_m$, $\beta\in H^{1,1} (X, \mathbb{R})$, then we have
\begin{align}\label{eq hodge index pos1}
  \left(\beta \cdot \alpha\cdot \alpha^{m-2} \cdot \omega^{n-m} \right)^2
  \geq \left(\beta^2 \cdot \alpha ^{m-2} \cdot \omega^{n-m} \right)\left(\alpha^2 \cdot \alpha^{m-2} \cdot \omega^{n-m}\right).
\end{align}
\end{lem}

\begin{proof}
Consider the quadratic form $$Q(\beta) = \left(\beta \cdot \alpha\cdot \alpha^{m-2} \cdot \omega^{n-m} \right)^2
  -\left(\beta^2 \cdot \alpha ^{m-2} \cdot \omega^{n-m} \right)\left(\alpha^2 \cdot \alpha^{m-2} \cdot \omega^{n-m}\right).$$
It is clear that for any $t\in \mathbb{R}$, we have $Q(\beta + t \alpha) = Q(\beta)$. Since $\alpha \in \Gamma_m$, $\beta + t \alpha$ falls in $\Gamma_m$ if $t>0$ is large enough. Applying Lemma \ref{lem hodge index pos} to $\beta + t \alpha \in \Gamma_m$, $Q(\beta + t \alpha) \geq 0$. This proves Lemma \ref{lem hodge index pos1}.
\end{proof}

As an immediate corollary of Lemma \ref{lem hodge index pos1}, we get:

\begin{cor}\label{hodge index}
Let $\alpha \in \Gamma_m$.
If $\beta \in H^{1,1}(X, \mathbb{R})$ satisfies $\beta \cdot \alpha^{m-1} \cdot \omega^{n-m} =0$, then
\begin{equation*}
  \beta^2 \cdot \alpha^{m-2} \cdot \omega^{n-m} \leq 0.
\end{equation*}
\end{cor}

\begin{proof}
By Lemma \ref{lem hodge index pos1}, we get the inequality. 
\end{proof}

This is a special case of Theorem \ref{thrm main result}.

\subsection{Proof of the main result}
Motivated by Corollary \ref{hodge index} and inspired by the mixed Hodge-Riemann bilinear relations in \cite{DN06,cattanimixedHRR}, it is natural to expect a mixed Hodge-index type theorem for $(1,1)$-classes, i.e., Theorem \ref{thrm main result}.

\begin{lem}\label{lem positive n-1}
Assume that $\omega \in \Lambda_{\mathbb{R}} ^{1,1}(\mathbb{C}^n)$ is a K\"ahler metric, and $\widehat{\alpha}_1,...,\widehat{\alpha}_{m-1}\in \Lambda_{\mathbb{R}} ^{1,1}(\mathbb{C}^n)$ are $m$-positive with respect to $\omega$, then
\begin{equation*}
\omega^{n-m}\wedge \widehat{\alpha}_1 \wedge...\wedge \widehat{\alpha}_{m-1}
\end{equation*}
is a strictly positive $(n-1, n-1)$-form.
\end{lem}

\begin{proof}
We only need to check that for any non-zero semi-positive $(1,1)$-form $\widehat{\beta}\in \Lambda_{\mathbb{R}} ^{1,1}(\mathbb{C}^n)$,
\begin{equation*}
\omega^{n-m}\wedge \widehat{\alpha}_1 \wedge...\wedge \widehat{\alpha}_{m-1}\wedge \widehat{\beta} >0.
\end{equation*}

To this end, note that $\overline{\Gamma}_n$ is the set of semi-positive $(1,1)$-forms and $\overline{\Gamma}_n\subset \overline{\Gamma}_m$, then applying Theorem \ref{thrm garding positive} yields the result.
\end{proof}

\begin{lem}\label{lem hodge index local}
In the same setting of Lemma \ref{lem positive n-1}, assume that $\widehat{\beta}\in \Lambda_{\mathbb{R}} ^{1,1}(\mathbb{C}^n)$ satisfies
\begin{equation*}
  \omega^{n-m}\wedge \widehat{\alpha}_1 \wedge...\wedge \widehat{\alpha}_{m-1}\wedge \widehat{\beta}=0,
\end{equation*}
then
\begin{equation*}
  \omega^{n-m}\wedge \widehat{\alpha}_1 \wedge...\wedge \widehat{\alpha}_{m-2}\wedge \widehat{\beta}^2 \leq 0
\end{equation*}
with equality holds if and only if $\widehat{\beta}=0$.
\end{lem}

\begin{proof}
This follows from Theorem \ref{thrm garding positive} and Theorem \ref{thrm garding ineq}.

Denote $\widehat{\Omega}=\omega^{n-m}\wedge \widehat{\alpha}_1 \wedge...\wedge \widehat{\alpha}_{m-2}$. Consider the polynomial $q(\widehat{\gamma})= \widehat{\Omega} \wedge \widehat{\gamma}^2$, then by G{\aa}rding's theory $q$ is complete and hyperbolic. Thus for any $\widehat{\beta}_1, \widehat{\beta}_2 \in \Gamma_m$,
\begin{equation*}
  q(\widehat{\beta}_1, \widehat{\beta}_2)^2 \geq q(\widehat{\beta}_1) q(\widehat{\beta}_1)
\end{equation*}
with equality holds if and only if $\widehat{\beta}_1, \widehat{\beta}_2$ are proportional. Similar to the arguments in Lemma \ref{lem hodge index pos1}, this can be generalized to the case when $\widehat{\beta}_1 \in \Gamma_m,  \widehat{\beta}_2 \in \Lambda_{\mathbb{R}} ^{1,1} (\mathbb{C}^n)$. When $\widehat{\Omega} \wedge \widehat{\beta}^2 =0$ and $\widehat{\Omega} \wedge \widehat{\alpha}_{m-1} \wedge \widehat{\beta} =0$, we get $\widehat{\beta} = c \widehat{\alpha}_{m-1}$, which in turns implies $c=0$, thus $\widehat{\beta} =0$. This finishes the proof.
\end{proof}

Now we give the proof of Theorem \ref{thrm main result}.

\begin{proof}[Proof of Theorem \ref{thrm main result}]
For the first part, if $\beta \in P^{1,1} (X, \mathbb{C})$, then its real part and imaginary part are in $P^{1,1} (X, \mathbb{R})$.
It is easy to see that it is sufficient to verify the negative definiteness on $P^{1,1} (X, \mathbb{R})$.

Take $\beta \in P^{1,1} (X, \mathbb{R})$. Let $\widehat{\beta}$ be a smooth representative of $\beta$, and let $\widehat{\alpha}_j$ be an $m$-positive representative of $\alpha_j$. Then $\beta \in P^{1,1} (X, \mathbb{R})$ is equivalent to
\begin{equation}\label{eq int 0}
  \int \widehat{\beta} \wedge \omega^{n-m}\wedge \widehat{\alpha}_1 \wedge ...\wedge \widehat{\alpha}_{m-1} =0.
\end{equation}
By Lemma \ref{lem positive n-1}, $\omega^{n-m}\wedge \widehat{\alpha}_1 \wedge ...\wedge \widehat{\alpha}_{m-1}$ is a strictly positive $(n-1, n-1)$-form at every point. It is also $d$-closed. Then (\ref{eq int 0}) guarantees that the Laplacian equation
\begin{equation}\label{eq laplacian}
  (\widehat{\beta}+dd^c \phi) \wedge \omega^{n-m}\wedge \widehat{\alpha}_1 \wedge ...\wedge \widehat{\alpha}_{m-1} =0
\end{equation}
always has a smooth solution $\phi$.

Write $\widehat{\beta}_\phi :=\widehat{\beta}+dd^c \phi$. Then (\ref{eq laplacian}) means that $\widehat{\beta}_\phi$ is a primitive $(1,1)$-form with respect to $\omega^{n-m}\wedge \widehat{\alpha}_1 \wedge ...\wedge \widehat{\alpha}_{m-1}$. By Lemma \ref{lem hodge index local}, we get
\begin{equation}\label{eq local hodge}
\widehat{\beta}_\phi ^2 \wedge \omega^{n-m}\wedge \widehat{\alpha}_1 \wedge ...\wedge \widehat{\alpha}_{m-2} \leq 0
\end{equation}
with equality holds if and only if $\widehat{\beta}_\phi =0$. Thus,
\begin{equation}\label{eq hodge index ineq}
  \beta^2 \cdot \omega^{n-m}\cdot \alpha_1 \cdot ...\cdot \omega_{m-2} =\int  \widehat{\beta}_\phi ^2 \wedge \omega^{n-m}\wedge \widehat{\alpha}_1 \wedge ...\wedge \widehat{\alpha}_{m-2}\leq 0.
\end{equation}

By (\ref{eq local hodge}), the equality in (\ref{eq hodge index ineq}) holds if and only if
\begin{equation}\label{eq eq case}
  \widehat{\beta}_\phi ^2 \wedge \omega^{n-m}\wedge \widehat{\alpha}_1 \wedge ...\wedge \widehat{\alpha}_{m-2} = 0
\end{equation}
at every point.
By (\ref{eq laplacian}), (\ref{eq eq case}) and Lemma \ref{lem hodge index local}, the equality in (\ref{eq hodge index ineq}) holds if and only if $\widehat{\beta}_\phi =0$, which is equivalent to $\beta =0$ in $H^{1,1} (X, \mathbb{R})$.

For the second part, for any $\gamma \in H^{1,1} (X, \mathbb{C})$, take $c \in \mathbb{C}$ such that
\begin{equation*}
  (\gamma - c \alpha_{m-1})\cdot  \alpha_{m-1} \cdot \Omega \cdot \omega^{n-m} =0.
\end{equation*}
Then $\gamma = (\gamma - c \alpha_{m-1}) + c \alpha_{m-1}$ is the desired Hodge decomposition.

For the last part, we only need to verify that the map is injective. Assume that $\beta$ satisfies $$\omega^{n-m}\cdot \Omega \cdot \beta =0,$$
then
\begin{equation*}
\omega^{n-m}\cdot \Omega \cdot \beta\cdot \overline{\beta} = \omega^{n-m}\cdot \Omega \cdot \alpha_{m-1} \cdot \beta=0.
\end{equation*}
Thus $\beta =0$ by the first part, which implies the injectivity of the map.

This finishes the proof of the theorem.
\end{proof}

Note that the Laplacian equation can be seen as a special case of the more general $dd^c$-equations used in \cite{DN06}.

\begin{rmk}
Assume that $\beta_1, \beta_2 \in H^{1,1} (X, \mathbb{R})$ are primitive, i.e., $\beta_i \cdot \alpha_{m-1}\cdot \Omega \cdot \omega^{n-m}=0$, then by Theorem \ref{thrm main result} the symmetric matrix \begin{equation*}
\left[
  \begin{array}{cc}
    \beta_1^2 \cdot \Omega\cdot \omega^{n-m} & \beta_1\cdot \beta_2\cdot  \Omega \cdot \omega^{n-m}\\
    \beta_1\cdot \beta_2\cdot  \Omega \cdot \omega^{n-m} &  \beta_2 ^2\cdot  \Omega \cdot \omega^{n-m} \\
  \end{array}
\right]
\leq 0
\end{equation*}
with the matrix degenerate if and only if $\beta_1, \beta_2$ are proportional.
\end{rmk}

\begin{rmk}\label{rmk strict conc}
By Theorem \ref{thrm main result}, the function $g(t)$ in Lemma \ref{lem concave ineq1} satisfies $g''(t)<0$ for any $t\geq 0$.
\end{rmk}

\begin{rmk}
In convex geometry setting, the corresponding Alexandrov-Fenchel inequalities have been established by \cite{guanpfAF}, where the authors made use of similar arguments of Alexandrov \cite{alexandrov1938} and also G{\aa}rding's theory for hyperbolic polynomials. Actually, this analogous result is also one of our motivation for Theorem \ref{thrm main result}.
\end{rmk}

\subsection{Log-concavity}
As an immediate consequence of Theorem \ref{thrm main result}, we get:

\begin{thrm}\label{thrm hyperbolic }
In the same setting of Theorem \ref{thrm main result}, the polynomial $Q(\beta):=\beta^2 \cdot \Omega \cdot \omega^{n-m}$ defined on $H^{1,1}(X, \mathbb{R})$ is complete and hyperbolic.
\end{thrm}

This implies the following form of Khovanskii-Teissier inequalities for $m$-positive classes.

\begin{prop}\label{KT m-posiitve}
Let $X$ be a compact K\"ahler manifold of dimension $n$, and let $\omega$ be a reference K\"ahler metric on $X$. Assume that $\alpha, \beta \in H^{1,1}(X, \mathbb{R})$ are $m$-positive with respect to $\omega$. Denote $a_k = \alpha^k \cdot  \beta^{m-k} \cdot \omega^{n-m}$, where $0\leq k \leq m$. Then the sequence $\{a_k\}$ is log concave, i.e., $$a_k ^2 \geq a_{k+1} a_{k-1}$$ for any $1\leq k \leq m-1$. Moreover, $a_k ^2 = a_{k+1} a_{k-1}$ for some $k$ if and only if $\alpha, \beta$ are proportional.
\end{prop}

\subsection{Miscellaneous}

\subsubsection{A form of mixed Hodge-Riemann bilinear relation} \label{sec HRR}

Let $X$ be a compact K\"ahler manifold of dimension $n$, and let $\omega$ be a K\"ahler metric. Assume that $\alpha_1, ..., \alpha_{m-p-q+1} \in \Gamma_m \subset H^{1,1}(X, \mathbb{R})$. We denote the class of $\omega$ by the same symbol. Denote $\Omega = \alpha_1 \cdot ... \cdot \alpha_{m-p-q}\cdot \omega^{n-m}$.

\begin{defn}
Let $\phi \in H^{p, q} (X, \mathbb{C})$, if $\phi\cdot  \alpha_{m-p-q+1}\cdot \Omega=0$, then we call $\phi$ primitive with respect to $\alpha_{m-p-q+1}\cdot \Omega$. The subspace of primitive $(p, q)$-classes is denoted by $P^{p, q} (X, \mathbb{C})$.
\end{defn}

\begin{defn}
On $H^{p, q} (X, \mathbb{C})$, the quadratic form $Q$ is defined by
\begin{equation*}
Q(\phi, \psi):=i^{q-p} (-1)^{(p+q)(p+q+1)/2} \phi \cdot \overline{\psi} \cdot \Omega.
\end{equation*}
\end{defn}

Motivated by Theorem \ref{thrm main result} and the results in \cite{DN06, cattanimixedHRR}, a more general Hodge-Riemann bilinear relation should hold, i.e., $Q$ is positive definite on $P^{p, q} (X, \mathbb{C})$. For this general relation, the hyperbolic polynomial tools employed in this paper do not work. We intend to discuss it elsewhere.

\subsubsection{Mixed Hessian equations}
As our motivation result, we have noted that a special case of our Hodge-index type inequalities (Corollary \ref{hodge index}) can be derived from complex Hessian equations. We ask whether the general case (Theorem \ref{thrm main result}) also follows from Hessian type equations.
It is related to the following equation, which looks interesting in itself. Let $\widehat{\alpha}, \widehat{\alpha}_1 ,..., \widehat{\alpha}_{m-l}$ be $d$-closed $m$-positive forms with respect to $\omega$, and let $\Phi$ be a positive smooth volume form, then
\begin{equation*}
  (\widehat{\alpha} + dd^c \phi)^l \wedge \widehat{\alpha}_1 \wedge...\wedge \widehat{\alpha}_{m-l} \wedge \omega^{n-m} = c \Phi.
\end{equation*}
should have a solution $\widehat{\alpha} + dd^c \phi$, which is $l$-positive with respect to $\omega$. The real analog of this kind of Hessian equations could also be applied to geometric inequalities.

\bibliography{reference}

\providecommand{\bysame}{\leavevmode\hbox to3em{\hrulefill}\thinspace}
\providecommand{\MR}{\relax\ifhmode\unskip\space\fi MR }
\providecommand{\MRhref}[2]{%
  \href{http://www.ams.org/mathscinet-getitem?mr=#1}{#2}
}
\providecommand{\href}[2]{#2}
\begin{thebibliography}{GMTZ10}

\bibitem[Ale38]{alexandrov1938}
Aleksandr~Danilovich Alexandrov, \emph{On the theory of mixed volumes of convex
  bodies {III}. extension of two theorems of {M}inkowski on convex polyhedra to
  arbitrary convex bodies.}, Mat. Sbornik \textbf{3} (1938), no.~45, 27--46
  (Russian).

\bibitem[Cat08]{cattanimixedHRR}
Eduardo Cattani, \emph{Mixed {L}efschetz theorems and {H}odge-{R}iemann
  bilinear relations}, Int. Math. Res. Not. IMRN (2008), no.~10, Art. ID
  rnn025, 20. \MR{2429243}

\bibitem[Dem93]{Dem93}
Jean-Pierre Demailly, \emph{A numerical criterion for very ample line bundles},
  J. Differential Geom. \textbf{37} (1993), no.~2, 323--374.

\bibitem[DK17]{dinewHessian}
S{\l}awomir Dinew and S{\l}awomir Ko{\l}odziej, \emph{Liouville and
  {C}alabi-{Y}au type theorems for complex {H}essian equations}, Amer. J. Math.
  \textbf{139} (2017), no.~2, 403--415. \MR{3636634}

\bibitem[DN06]{DN06}
Tien-Cuong Dinh and Vi{\^e}t-Anh Nguy{\^e}n, \emph{The mixed {H}odge-{R}iemann
  bilinear relations for compact {K}\"ahler manifolds}, Geom. Funct. Anal.
  \textbf{16} (2006), no.~4, 838--849.

\bibitem[DP04]{DP04}
Jean-Pierre Demailly and Mihai P{\u{a}}un, \emph{Numerical characterization of
  the {K}\"ahler cone of a compact {K}\"ahler manifold}, Ann. of Math. (2)
  \textbf{159} (2004), no.~3, 1247--1274.

\bibitem[Gar59]{gardinghyperbolic}
Lars Garding, \emph{An inequality for hyperbolic polynomials}, J. Math. Mech.
  \textbf{8} (1959), 957--965. \MR{0113978}

\bibitem[GMTZ10]{guanpfAF}
Pengfei Guan, Xi-Nan Ma, Neil Trudinger, and Xiaohua Zhu, \emph{A form of
  {A}lexandrov-{F}enchel inequality}, Pure Appl. Math. Q. \textbf{6} (2010),
  no.~4, Special Issue: In honor of Joseph J. Kohn. Part 2, 999--1012.
  \MR{2742035}

\bibitem[Gro90]{gromov1990convex}
Misha Gromov, \emph{Convex sets and {K}\"ahler manifolds}, Advances in
  Differential Geometry and Topology, ed. F. Tricerri, World Scientific,
  Singapore (1990), 1--38.

\bibitem[Hor94]{hormanderConvexity}
Lars Hormander, \emph{Notions of convexity}, Progress in Mathematics, vol. 127,
  Birkh\"{a}user Boston, Inc., Boston, MA, 1994. \MR{1301332}

\bibitem[Voi07]{voisinHodge1}
Claire Voisin, \emph{Hodge theory and complex algebraic geometry. {I}}, english
  ed., Cambridge Studies in Advanced Mathematics, vol.~76, Cambridge University
  Press, Cambridge, 2007, Translated from the French by Leila Schneps.
  \MR{2451566}

\bibitem[Yau78]{Yau78}
Shing~Tung Yau, \emph{On the {R}icci curvature of a compact {K}\"ahler manifold
  and the complex {M}onge-{A}mp\`ere equation. {I}}, Comm. Pure Appl. Math.
  \textbf{31} (1978), no.~3, 339--411.

\end{thebibliography}
\bibliographystyle{amsalpha}

\bigskip

\bigskip

\noindent
\textsc{Tsinghua University, Beijing, China}\\
\noindent
\verb"Email: jianxiao@mail.tsinghua.edu.cn"

\end{document}